\documentclass{lmcs} 
\pdfoutput=1

\usepackage{lastpage}
\lmcsdoi{17}{2}{21}
\lmcsheading{}{\pageref{LastPage}}{}{}%
{May~29,~2019}{May~27,~2021}{}

\usepackage[utf8]{inputenc}

\theoremstyle{plain} 
\usepackage[all]{xy}
\usepackage{latexsym,amssymb,amsfonts,amsthm,amsmath}
\newcommand{\mlu}{\mbox{\bf MLtt$_1$}}

\newtheorem{theorem}[thm]{Theorem}
\newtheorem{lemma}[thm]{Lemma}

\theoremstyle{definition}
\newtheorem{definition}[thm]{Definition}
\newtheorem{remark}[thm]{Remark}

\newcommand{\mluif}{\mbox{\bf MLtt$_{ind}$}}
\newcommand{\mlur}{\mbox{\bf MLtt$_1^r$}}
\newcommand{\mluifr}{\mbox{\bf MLtt$_{ind}^r$}}

\newcommand{\mtt}{\mbox{{\bf mTT}}}
\newcommand{\mttind}{{\bf mTT}$_{ind}$}

\newcommand{\emtt}{\mbox{{\bf emTT}}}
\newcommand{\emttind}{\mbox{{\bf emTT$_{ind}$}}}
\newcommand{\ct}{\mbox{\bf CT}}
\newcommand{\ac}{\mbox{\bf AC}}
\newcommand{\tar}{\widehat{ID_1}}
\newcommand{\czfrea}{{\bf CZF+REA}}

\newcommand{\iduh}{$\widehat{ID}_1$}

\newcommand{\mf}{{\bf MF}}
\newcommand{\mfind}{{\bf MF$_{ind}$}}
\def\cov{\lhd} 

\newcommand{\rft}{\widetilde{\mathsf{rf}}}
\newcommand{\trt}{\widetilde{\mathsf{tr}}}
\newcommand{\sigmat}{\widetilde{\Sigma}}
\newcommand{\pit}{\widetilde{\Pi}}
\newcommand{\pp}{\mathfrak p}

\begin{document}

\title[Realizability for formal topologies, Church's Thesis, Axiom of Choice]{A  realizability semantics for inductive formal topologies, Church's Thesis and Axiom of Choice}
  
  \author[M.~E.~Maietti]{Maria Emilia Maietti\rsuper{a}}
\address{\lsuper{a}Dipartimento di Matematica ``Tullio Levi Civita'',
Universit\`a di Padova, Italy}
\email{maietti@math.unipd.it}
\email{maschio@math.unipd.it}

\author[S.~Maschio]{Samuele Maschio\rsuper{a}}

\address{\lsuper{b}School of Mathematics,
University of Leeds, UK}
\email{M.Rathjen@leeds.ac.uk}
\author[M.~Rathjen]{Michael Rathjen\rsuper{b}}

\thanks{Projects EU-MSCA-RISE project 731143 "Computing with Infinite
Data" (CID), MIUR-PRIN 2010-2011 and Correctness by Construction (EU
7th framework programme, grant no.~PIRSES-GA-2013-612638) provided
support for the research presented in the paper.}

\maketitle
\begin{abstract}
  We present a Kleene realizability semantics for the intensional level of the
  Minimalist Foundation, for short \mtt, extended with inductively generated
  formal topologies, the formal Church's thesis and axiom of choice.
  
  This semantics is an extension of the one used to show  the consistency of the intensional level of the
  Minimalist Foundation with the axiom of choice and the formal Church's thesis in the work by Ishihara, Maietti, Maschio, Streicher
  [Arch.Math.Logic,57(7-8):873-888,2018].

  A main novelty here is that such a semantics is formalized in a constructive
  theory as Aczel's constructive set theory {\bf CZF} extended with the 
  regular extension axiom.
\end{abstract}
\section{Introduction}
A  main motivation for introducing the Minimalist Foundation, for short \mf, in \cite{mtt,m09}
was the desire to provide a foundation where to formalize constructive point-free topology in a way compatible with most relevant constructive foundations.
In particular, \mf\ was designed with the purpose of formalizing the topological results developed
by adopting the approach of Formal Topology by P. Martin-L{\"o}f and G. Sambin 
introduced in \cite{S87}. This approach was further enriched with the introduction of Positive Topology by Sambin in \cite{somepoint}.  A  remarkable novelty of this approach to constructive topology
was the advent of  inductive  topological methods (see
\cite{CSSV03,cms13}) to represent the point-free topologies  of the real number line, 
 of Cantor space and of Baire space.

However, while the basic notions of Formal Topology 
can be formalized in the Minimalist Foundation in \cite{m09},  not all the constructions of  inductively generated formal topologies  are formalizable there.
For instance,  the formal topology of Cantor space and that of the real line are formalizable in \mf\  thanks to some characterizations given in \cite{silviobar},
but the formal topology of the Baire space is not thanks to results in \cite{CR}.

In any case this limitation is not a  surprise  since \mf\ was introduced to be a minimalist foundation compatible with the most relevant
constructive and classical foundations for mathematics in the literature in the sense of being interpretable in them by preserving
the meaning of logical and set-theoretic onstructors.

To better fulfill this requirement  \mf\  was  ideated in \cite{mtt} as a  two level system which was completed in \cite{m09}.
More precisely the two-level structure of \mf\ consists of 
an intensional level, called \mtt\ and based on an intensional type theory \`a la Martin-L{\"o}f,
aimed at exhibiting the computational contents of ma\-the\-ma\-tical
proofs, and 
an extensional level formulated in a language as close as possible to that
of present day mathematics which is interpreted in the intensional level by means
of a quotient model (see \cite{m09}).
In particular, the intensional level \mtt\ is quite weak in proof-theoretic strength because it can be  interpreted in the  fragment of Martin-L{\"o}f's type
theory with one universe, or directly in Feferman's theory of non-iterative fixpoints \iduh\ as first shown in \cite{ms16}.

The  authors  in \cite{mtt} were led to propose a two-level system as a notion of constructive foundation for the following reason.
They wanted to found   constructive mathematics in a  system  interpretable in an extension of Kleene realizability semantics
of intuitionistic arithmetics with finite types. The  reason is that  such a semantics makes evident the extraction of witnesses from existential statements
in a computable way thanks to the validity of the axiom of choice ($\ac$) and the formal Church's thesis ($\ct$).
In more detail,  $\ac$ states
 that  from any total relation
we can extract a type-theoretic function as follows:
\[\tag{\bf AC}(\forall x\in A)\,(\exists y\in B)\,R(x,y)\rightarrow (\exists f\in A\rightarrow B)\,(\forall x\in A)\,R(x,\mathsf{Ap}(f,x))\]
with $A$ and $B$ generic collections and $R(x,y)$ any relation,
while $\ct$ (see also ~\cite{DT88}) states that from any total relation on natural numbers
we can extract a (code of a) recursive function by using the Kleene predicate $T$ and the extracting function $U$
 \[\tag{\bf CT} (\forall f\in \mathsf{N}\rightarrow\mathsf{N})(\exists e\in \mathsf{N})\,(\forall x\in \mathsf{N})\,(\exists z\in \mathsf{N})\,(T(e,x,z)\wedge \mathsf{Ap}(f,x)=_\mathsf{N}U(z)).\] 
As a consequence, their desired  constructive foundation should have been consistent  with $\ac$ and $\ct$. But this consistency requirement is a very strong property since it rules out the validity of extensional principles used in everyday practice of mathematics,
including the extensional equality of functions, beside not being satisfied by most constructive foundations in the literature. Therefore, the authors of \cite{mtt}  ended up in defining a constructive foundation as a {\it two-level system} consisting of an  {\it extensional level}, formulated
in a language close to that 
of informal mathematics and
validating  all the desired extensional properties, and of an {\it intensional level}
consistent with the axiom of choice $\ac$ and  $\ct$ where the extensional level is interpreted
via a quotient model. The system \mf\ is an example of this notion of  two-level constructive foundation. The proof that its intensional level is consistent with $\ac$ and $\ct$ was given
only recently in \cite{IMMS}.

The purpose of this paper is two-fold.
First we present an extension \mfind\ of \mf\ with the inductive 
definitions sufficient to inductively generate  formal topologies and necessary to define  
  {\it inductively generated basic covers}  which constitute a predicative presentation of suplattices (see \cite{batsampre,cms13}). 
This is motivated by the fact that in \cite{CSSV03}
the problem of generating
formal topologies inductively is reduced to that of generating
suplattices in an inductive way. As \mf\ also \mfind\ is presented as a two-level system obtained by extending each level of \mf\
with rules generating basic covers inductively.  In particular, the  rules of inductively generated basic covers added to \mtt\ to form  the intensional level \mttind\  are driven  by those of well-founded sets in Martin-L{\"o}f's type theory in \cite{PMTT} without assuming generic well-founded sets or ordinals as done for the formalization  of such covers within type theory 
in \cite{CSSV03,silviobar}. 

Then our main purpose is to show that the extension  \mfind\  of \mf\  is also an example of the mentioned notion of two-level constructive foundation in \cite{mtt}. To this aim the key result shown here is
 that
 its intensional level \mttind\ is consistent with $\ac$ and $\ct$.

In order to meet our goal we produce a  realizability semantics
for \mttind\ by extending the one used to show the consistency of the intensional level of \mf\  with \ac+\ct\  in \cite{IMMS}, which in turn extends Kleene realizability interpretation of intuitionistic arithmetic.

A main novelty of the semantics in this paper is that it is formalized
in a constructive theory as the (generalized) predicative  set theory  \czfrea, namely Aczel's constructive Zermelo-Fraenkel set theory extended with the regular extension axiom  {\bf REA}.

To this purpose it is crucial to modify the realizability interpretation
in \cite{IMMS} in the line of the  realizability interpretations of Martin-L\"of type theories in extensions of Kripke-Platek set theory introduced 
 in \cite{R93} (published as \cite{RG94}).

Therefore, contrary to the semantics in \cite{IMMS}, which was formalized in a classical theory
as Feferman's theory of non-iterative fixpoints \iduh, here
we produce a proof that \mttind, and hence $\mtt$, is {\it constructively}
consistent with \ac+\ct.


 As in \cite{IMMS}, we actually build a realizability model for
a fragment of  Martin-L{\"o}f's type theory~\cite{PMTT}, called \mluif, 
where \mttind\ extended with the axiom of choice can be easily interpreted.

As it turns out, $\mathbf{CZF}+\mathbf{REA}$ and \mluif\ possess the same proof-theoretic strength.

In the future we intend to further extend our realizability to model
\mttind\ enriched with coinductive definitions to represent Sambin's
generated Positive Topologies.
Another possible line of investigation would be to employ our realizability
semantics to establish the consistency strength of \mttind\ or  of the extension of \mtt\ with particular
inductively generated topologies, like that of the Baire space.

\section{The extension \mfind\  with inductively generated   formal topologies}\label{mfind}
Here we describe the extension \mfind\ of  \mf\
capable of formalizing all the
examples of formal topologies defined by inductive methods
introduced in \cite{CSSV03}.

In that paper, the problem of generating the minimal 
formal topology which satisfies some given axioms
is reduced to show how to generate a complete suplattice
in terms of an infinitary relation called {\bf basic cover relation}
$$a\cov V$$
between elements $a$ of a set $A$, thought of as {\it basic opens}, and subsets $V$ of $A$, meaning that {\it the basic open $a$ is covered by
    the union of basic opens in the subset $V$}.

Then the elements of the generated suplattice would be fixpoints of the associated {\it closure operator}
$$\cov(-) : \mathcal{ P}(A)\ \longrightarrow \  \mathcal{ P}(A)$$
defined by putting
$$\cov(V)\, \equiv\, \{ \ x\in A\ \mid\ x\cov
V\ \}$$
These suplattices are {\it complete} with respect to families of subsets {\it indexed over a set}.

Furthermore,  a formal topology is defined as a basic cover relation
satisfying a convergence property and a positivity predicate (see \cite{CSSV03,mv04,batsampre,cms13}). Indeed in this case the resulting complete suplattice of $\cov$-fixpoints actually forms
a {\it predicative locale} which is {\it overt} (or {\it open} in the original terminology by Joyal and Tierney) for the presence of the positivity predicate.

The tool of  basic covers appears to be the only one available in the literature to represent  complete suplattices in
 most-relevant  predicative constructive foundations 
including Aczel's {\bf CZF}, Martin-L{\"o}f's type theory and also \mf.

The reason is that
{\it there exist no non-trivial examples of complete suplattices that form a set} in such predicative foundations
(see \cite{Cu10}). As a consequence, there exist no non-trivial examples
of locales which form a set and the approach of formal topology based
on a cover relation seems to be compulsory (see also \cite{whyp})
when developing topology in a constructive predicative foundation, especially
in \mf.

In \cite{CSSV03} it 
was introduced a method for generating basic covers inductively
starting from {\it an indexed set of axioms}, called {\it axiom set}.
Such a method  allows
to generate  a formal topology inductively when the
basic  cover relation $\cov$  is defined on  {\it a preordered set $(A, \leq)$ }
and it is generated by an axiom set   satisfying a so called {\it localization condition} which refers to the preorder defined on $A$.
 A general  study of the relation between basic covers and formal covers including their inductive generation is given in \cite{cms13}.\\

In the following we describe
a suitable extension of \mf\ capable of representing inductively generated
basic covers, and hence also formal topologies.

We start by describing how to enrich the extensional level  \emtt\ of \mf\ in \cite{m09} with such inductive basic covers.
The reason is that the language of \emtt\  is more apt to represent the topological axioms given that it is very close to that of everyday mathematical
practice (with proof-irrelevance of propositions and an encoding of the usual language of  first-order arithmetic and  of subsets of a set, see \cite{m09}).

We recall that in \emtt\ we have four kinds of types, namely
{\bf collections}, {\bf sets},  {\bf propositions} and {\bf small propositions}
according to the following subtyping relations:
$$
\xymatrix@C=4em@R=1em{
 {\mbox{\bf small propositions}}\ar@{^{(}->}[dd]\ar@{^{(}->}[rr]&  &
 {\mbox{\bf sets}}\ar@{^{(}->}[dd] \\ 
&&\\
  {\mbox{\bf propositions}}\ar@{^{(}->}[rr] && {\mbox{\bf collections}}
}
$$
where collections  include the power-collection
${\mathcal P}(A)$ (which is not  in general a set!)  of  any set $A$ and small propositions are defined as those propositions closed
under intuitionistic connectives, propositional equality and quantifiers restricted to sets.

We first  extend \emtt\ with new primitive {\it small propositions}
$$a\triangleleft_{I,C} V\ prop_s$$
 expressing that {\it the basic open $a$ is covered by the union of basic opens in $V$}
 for any $a$ element of a {\it set} $A$, $V$ subset of $A$,
 assuming that the basic  cover is
generated by a family
of (open) subsets  of $A$
indexed on  a family of  sets $I(x)\ set \ [x\in A]$ and represented by
  $$C(x,j) \in {\mathcal P}( A)\ [x\in A, j\in I(x)]. $$

Recall from \cite{m09} that in \emtt\ we can define a {\it subset membership}
$$a\,\epsilon\, V$$
between a subset $V\in  {\mathcal P}( A)$ and an element $a$  of a set $A$.
Note that while the membership $a\in A$ is {\it primitive} in \emtt\
 and expresses that $A$ is the type of $a$,  the membership $\epsilon$ essentially expresses a relation
 between subsets of $A$ and its elements.

 The precise rules extending \emtt\ to form a new type system \emttind\ are the following:

  {\bf Rules of inductively generated basic covers in \emttind}
  $$\begin{array}{l}
\mbox{\rm F-$\triangleleft$} \
\displaystyle{\frac{\begin{array}{l}
      A \ set \ \ \ \  I(x)\  set \  [x\in A] \ \ \ C(x,j) \in {\mathcal P}( A)\ \ [x\in A, j\in I(x)]\\
      V\in {\mathcal P}( A)\qquad a\in A\ \end{array} }
  {\displaystyle a\triangleleft_{I,C} V\ prop_s }}\\[20pt]
\mbox{\rm rf-$\triangleleft$} 
\displaystyle{\frac{\begin{array}{l}
      A \ set \ \ \ \  I(x)\  set \  [x\in A] \ \ \ C(x,j)\in {\mathcal P}( A) \ \ [x\in A, j\in I(x)]\\
   V\in {\mathcal P}( A)
   \qquad\qquad   a\,\epsilon\, V\ true \end{array} }
  {\displaystyle a\triangleleft_{I,C} V\ true}}\\
\\
\mbox{\rm tr-$\triangleleft$} \
\displaystyle{
  \frac{\begin{array}{l} A\ set \ \ \ \  I(x)\  set \  [x\in A] \ \ \ C(x,j)\in {\mathcal P}( A)  \ \ [x\in A, j\in I(x)]\\
      a\in A\qquad i\in I(a)\qquad  \qquad  V\in {\mathcal P}( A)\\
       ( \forall_{y\in A} ) \  (\  y\, \epsilon\, C(a,i)\  \rightarrow
         \  y\triangleleft_{I,C} V\ )\ true\end{array} }
  {\displaystyle a\triangleleft_{I,C} V \  true }}
\\[20pt]
\mbox{ind-$\triangleleft$} 
\displaystyle{\frac{ 
    \begin{array}{l}    A\ set \ \ \ \  I(x)\  set \  [x\in A] \ \ \
      C(x,j)\  \in {\mathcal P}( A)\ \ [x\in A, j\in I(x)]\\
           P(x)\  prop\ [x\in A]\qquad V\in {\mathcal P}( A) \qquad
        \mathsf{cont}(V,P)\ true\ \\
a\in A\qquad a \cov_{I,C} V\ true \end{array} }{ P(a)\ true\       }}
\end{array}$$
    where
    $$\begin{array}{rl}
      \mathsf{cont}(V,P)\ \equiv\ & \forall_{x\in A}\ (\ x\,\epsilon\, V\ \rightarrow \ P(x) \ )\ \\
      & \qquad \wedge\ \forall_{x\in A}\ \ (\ \forall_{j \in I(x)}\  \forall_{y\in A}\ (y\,\epsilon \, C(x,j)\ \rightarrow \ P(y) \ ) \ \rightarrow\ P(x)\ ) 
    \end{array}$$
    \\

    \noindent where above we adopted
    the convention of writing $\phi\ true$ for a proposition $\phi$
    instead of $\mathsf{true}\in \phi$ as in \cite{m09}.

The first rule expresses the formation of the small propositional function representing a basic cover, the second one expresses a form of reflexivity since it states that any element of a subset $V$ is covered by it,
 the third one expresses a form of transitivity property  applied to open subsets of the axiom set and the fourth one is a form of induction on the generated basic cover.

    A main example of inductively generated cover formalizable in \emttind\  is   that of  the topology of the {\it real line}, represented by
    Joyal's inductive formal cover $\cov_{r}$ of Dedekind real numbers defined
    on the set $ \mathbb Q\, \times\, \mathbb Q$ 
    (which acts as $A$ in the rules above) where $\mathbb Q$ is the set of {\it rational numbers}. This formal cover is generated by a family of  open subsets 
$C(\langle p,q\rangle ,  j )$ for $\langle p,q\rangle \in \mathbb Q\, \times\, \mathbb Q$ and  $j\in I(\langle p,q\rangle)$   where  $I(\langle p,q\rangle)$  is the set of indexes of the following rules:\\
   
\begin{center}
    \begin{tabular}{cc}
  ${\displaystyle\frac{q\leq p}{\langle p,q\rangle\cov_{r} U}}$
\qquad &
${\displaystyle\frac{p'\leq p<q\leq q' \qquad \langle p',q'\rangle\cov_{r} U}
                     {\langle p,q\rangle\cov_{r} U}}$\\[20pt]
 \qquad ${\displaystyle
\frac{p\leq r < s\leq q\qquad \langle p,s\rangle\cov_{r} U \quad \langle r,q\rangle\cov_{r} U }
     {\langle p,q\rangle\cov_{r} U}}$ & $ \ {\displaystyle\frac {\mathsf{wc}(\langle p,q\rangle)\cov_{r} U}{\langle p,q\rangle\cov_{r} U}}$
\end{tabular}
    \end{center}

\noindent
where in the last axiom we have used the abbreviation
$$
\mathsf{wc}(\langle p,q\rangle)\equiv \{\, \langle p',q'\rangle\in \mathbb Q\, \times\, \mathbb Q\ \mid\  p<p'< q'<q\}
$$ 
($\mathsf{wc}$ stands
for `well-covered').  

Some relevant applications 
regarding the formal topology of the real line 
are discussed for instance in \cite{heine,Palmgren05continuityon,whyp}.
\\

There exists also an alternative presentation of the formal topology of the real line due to T. Coquand
in \cite{heine}  which makes this topology formalizable in \emtt, and hence in \mf. The reason is that  its cover can be defined in terms of  another inductive cover generated by a finite set of axioms  and  formalizable in \emtt\  by using a characterization of inductive formal topologies given in \cite{silviobar} and  the crucial fact that  in \emtt\ we can define sequences of subsets
by recursion on natural numbers.

Other relevant examples of inductively generated formal topologies are the topologies of Cantor and Baire spaces.
These are instances of the more general notion of  {\it tree formal topology relative to a set $E$} represented by  the cover $\triangleleft_{tr(E)}$
on the set  $\mathsf{List}(E)$ of lists of $E$   generated by a family of open subsets $C( l,  j )$ for a list $l$ and $j\in I(h)$ where $I(h)$ is the set of indexes of the following rules:
\\

\begin{center}
    \begin{tabular}{l}
$
\displaystyle{\frac{ s\sqsubseteq l    \qquad l\triangleleft_{tr(E)} V  }
{\displaystyle   s\triangleleft_{tr(E)}  V }}$
\qquad
$
\displaystyle{\frac{ \forall{x\in E}\ \ [l,x]\triangleleft_{tr(E)} V }
{\displaystyle l\triangleleft_{tr(E)} V }}$
\end{tabular}
\end{center}

\noindent
where  $s\sqsubseteq  l$  means that the list $l$ is an initial segment of  the list $s$, formally defined as
$s\sqsubseteq  l\,\equiv\, \exists_{t\in \mathsf{List}(E)}\, 
\mathsf{Id} (\mathsf{List}(E), s, [ l , t] )$ and $[l,t]$ is the concatenation of the list $l$ with $t$.

Then  the {\it formal  topology of the Cantor space} is the tree topology $\triangleleft_{tr(\{0,1\} )}$ 
 when $E$ is the boolean set $\{0,1\}$ and the {\it formal topology of  the Baire space} is the tree topology $\triangleleft_{tr( \mathsf{N})}$
when $E$ is the set  $\mathsf{N}$ of natural numbers.

The formal topology of the  Baire space is a genuine example of inductively generated cover  definable in \emttind\ but not in \emtt\ contrary to that  of the  Cantor space which is definable in \emtt:
\begin{prop}
The formal topology of the  Baire space $\triangleleft_{tr( \mathsf{N})}$ is not formalizable in \emtt, and hence in \mf,
while the formal  topology of Cantor space $\triangleleft_{tr(\{0,1\} )}$ is formalizable there. \end{prop}
 \begin{proof}
 The formal topology of the Cantor space  $\triangleleft_{tr(\{0,1\} )}$   is formalizable in \emtt\ 
thanks to the characterization  of  $\triangleleft_{tr(\{0,1\} )} (V)$ for a subset $V$ of  $\mathsf{List}(\{0,1\})$  in \cite{silviobar}, beside the fact that in \emtt\
we can define sequences of subsets of a set by recursion on natural numbers.

Instead the formal  topology of the Baire space $\triangleleft_{tr( \mathsf{N})}$  is not formalizable in \emtt\  for the following reason.
From \cite{m09} we know that \emtt\ can be translated in  $\mathbf{CZF}$ by preserving the meaning of propositions and sets.
Now if the topology $\triangleleft_{tr( \mathsf{N})}$  were formalizable in \emtt\ then it would be formalizable also in  $\mathbf{CZF}$ but from \cite{CR} we know that this is not possible.
\end{proof}

More generally, all the finitary topologies in \cite{CSSV03} are definable  in \emtt, and hence in \mf,  in a way similar to what done to formalize the
 formal topology of the Cantor space.
\\

It is worth noting that different presentations of basic covers
may yield to the same complete suplattice.
For example,  any complete suplattice 
presented by  (the collection of fixpoints
associated to) a basic cover $\cov_{I,C}$  on  a quotient set  $B/R$,
can be equivalently presented by   a cover on the set $B$ itself
which behaves like $\cov_{I,C}$  but in addition 
it considers as equal opens those elements which are related by $R$.

In order to properly show this fact, which will motivate some definitions in the next section,
we define a correspondence between subsets of $B/R$
    and subsets of $B$ as follows:
    \begin{definition}
      In \emttind, 
      given  a quotient set $B/R$, for any subset $W\in \mathcal{P}(B/R)$
      we define
      $$\mathsf{es}(W)\, \equiv \, \{\ b\in B\ \mid\  [b]\,\epsilon\, W\     \}$$
        and given any $V\in  \mathcal{P}(B)$ we define $\mathsf{es}^{-}(V)\, \equiv \, \{\ z\in B/R\ \mid\  \exists_{b\in B}\ (\ b\,\epsilon\, V\ \wedge\ z=_{B/R}[b] \ )\}$.
    \end{definition}
    \begin{definition}\label{indquo}
Given an axiom set  represented by a set $A\,\equiv\, B/R$ with
$ I(x)\  set\  [x\in A]$ and $  C(x,j)\in {\mathcal P}(A) \ \ [x\in A, j\in I(x)]$, we define a new axiom set as follows:
$$\begin{array}{l}
 A^R\, \equiv\, B\qquad \qquad I^R(x)\, \equiv\, I([x]) + \ (\Sigma y\in B)\ R(x,y) \qquad \qquad \mbox{ for } x\in B\\
\end{array}
$$
where $C^R(b,j)$ is the formalization of
$$C^R(b,j)\, \equiv\, \begin{cases}
  \mathsf{es}(\, C([b],j)\, ) &  \mbox{ if } j\in  I([b])\\
  \{\, \pi_1(j) \, \} &  \mbox{ if } j\in (\Sigma y\in B)\ R(b,y) \\
\end{cases}$$
for $b\in B$ and $j\in I^R(x)$.

\noindent We then call $\cov_{I,C}^R$ the inductive basic cover generated from this axiom set.
      \end{definition}
    It is then easy to check that
    \begin{lemma}\label{eqcov}
      For any axiom set in \emttind\ represented by a set $A\,\equiv\, B/R$ with
      $ I(x)\  set\  [x\in A]$ and $  C(x,j)\in {\mathcal P}(A) \ \ [x\in A, j\in I(x)]$, the suplattice defined by $\cov_{I,C}$ is isomorphic
      to that defined by $\cov_{I,C}^R$ by means of an isomorphism of suplattices.
\end{lemma}
        \begin{proof}
        It is immediate to check that for any subset  $W$ of $B/R$
        which is a fixpoint for $\cov_{I,C}$ the subset
        $\mathsf{es}(W)$ is a fixpoint for $\cov_{I,C}^R$ and that, conversely,
        for any subset  $V$ of $B$
        which is a fixpoint for $\cov_{I,C}^R$ the subset 
        $\mathsf{es}^{-}(V)$ is a fixpoint for $\cov_{I,C}$. Moreover, this correspondence
        preserves also the suprema defined as in \cite{cms13}.
        Alternatively, one could check that the relation
        $z\, F\, b\, \equiv\  \mathsf{Id} (\, B/R\, ,\,  z\, ,\, [b]\, )$, namely the propositional equality of $z$ with $[b]$,
        defines a basic cover isomorphism in the sense of \cite{cms13}
        between the basic cover $\cov_{I,C}$ and $\cov_{I,C}^R$.\end{proof}

\subsection{The intensional level \mttind}
Here we describe the extension \mttind\ of the intensional level \mtt\ of \mf\
capable of interpreting the extension \emttind.

We recall that in \mtt\ we have the same four kinds
of types as in \emtt\  with the difference that  in \mtt\ power-collections of sets
are replaced by the existence of a  {\it collection of small propositions} $\mathsf{prop_s}$
and function collections $A\rightarrow \mathsf{prop_s}$ for any set $A$.
Such collections are enough to interpret power-collections of sets in \emtt\  within a quotient model of dependent extensional types
built over \mtt, as explained
in \cite{m09}.

In order to complete \emttind\  into  a two-level foundation according to the requirements in \cite{mtt}, we need to define an extension of \mtt\ with
a proof-relevant version of 
 the inductively generated basic covers of \emttind.
 
 To this purpose we defined the extension \mttind\ by extending  \mtt\ in \cite{m09}  with  new small propositions
$$ a\triangleleft_{I,C} V\ prop_s$$
and corresponding  new proof-term constructors associated to them so that
judgements asserting that some proposition is true  in \emttind\ 
are turned into   judgements of \mttind\
producing a proof-term of the corresponding proposition.

We recall that  in \mttind\ as in \mtt\ in \cite{m09}  the universe of small propositions  is defined  in the version  \`a la Russell.
A version  of \mtt\ with the universe of small propositions  \`a la Tarski can be found  in \cite{ms16}. 

 It is worth noting that the
equality rules of the inductive basic covers 
are driven  by those of well-founded sets in Martin-L{\"o}f's type theory in \cite{PMTT} without assuming generic well-founded sets or ordinals as in the representations given
in \cite{CSSV03,silviobar}. However, in accordance
with the idea that proof-terms of propositions of \mtt\ represent just
a constructive rendering of the proofs of propositions in \emtt, we do restrict the elimination rules
of inductive basic covers
to act toward propositions which do not depend on their proof-terms, since these
proof-terms do not appear at the extension level  \emtt.

When expressing the rules of inductive basic covers
 we use the abbreviation $$ a\,\epsilon\, V \qquad \mbox{ to mean }\qquad \mathsf{Ap}(\, V\, , \, a\, )$$
for any set $A$,  any small propositional function $ V\in A\rightarrow\mathsf{prop_s}$ and any element $a\in A$.
\\

The precise rules of inductive basic covers extending \mtt\ to form a new type system \mttind\ are the following:

  {\bf Rules of inductively generated basic covers in \mttind}
  \\
  $$\begin{array}{l}
\mbox{\rm F-$\triangleleft$} \
\displaystyle{\frac{\begin{array}{l}
      A \ set \ \ \ \  I(x)\  set \  [x\in A] \ \ \ C(x,j) \in A\rightarrow 
      \mathsf{prop_s}\ [x\in A, j\in I(x)]\\
        V\in A\rightarrow\mathsf{prop_s}\qquad a\in A \end{array} }
  {\displaystyle a\triangleleft_{I,C} V\ prop_s }}\\[20pt]
\mbox{\rm rf-$\triangleleft$} \
\displaystyle{\frac{\begin{array}{l}
      A \ set \ \ \ \  I(x)\  set \  [x\in A] \ \ \ C(x,j) \in A \rightarrow \mathsf{prop_s} \ \ [x\in A, j\in I(x)]\\
        V\in A \rightarrow \mathsf{prop_s}
   \qquad a\in A\qquad r\in   a\,\epsilon\, V\ \end{array} }
  {\displaystyle \mathsf{rf}(a,r)\in a\triangleleft_{I,C} V}}
\\[20pt]
\end{array}$$
  $$\begin{array}{l}
\mbox{\rm tr-$\triangleleft$} \
\displaystyle{
  \frac{\begin{array}{l} A\ set \ \ \ \  I(x)\  set \  [x\in A] \ \ \ C(x,j) \in  A \rightarrow \mathsf{prop_s}\ [x\in A, j\in I(x)]\\
        V\in  A \rightarrow \mathsf{prop_s} \qquad\ \
      a\in A\qquad i\in I(a)\\
        r\in \forall_{x\in A}\  (\ x\,\epsilon\, C(a,i)\rightarrow 
      x\triangleleft_{I,C} V\ )  \ \ \end{array} }
  {\displaystyle \mathsf{tr}(a,i,r)\in a\triangleleft_{I,C}V \  }}
\\[20pt]

\mbox{ind-$\triangleleft$} 
\displaystyle{\frac{ 
    \begin{array}{l}    A\ set \ \ \ \  I(x)\  set \  [x\in A] \ \ \
      C(x,j)\  \in A\ \rightarrow \mathsf{prop_s}\ \ [x\in A, j\in I(x)]\\[5pt]
         P(x)\  prop\ [x\in A]\qquad  V\in A\ \rightarrow\  \mathsf{prop_s}\qquad \\[5pt]
      a\in A\qquad   m\in a \cov_{I,C} V\qquad \\[5pt]
      q_1(x,z)\in P(x)\ [ x\in A, z\in x\,\epsilon\, V]\qquad \qquad \\[5pt]
    q_2(y,j, f) \in P(y)\  [ y\in A, j\in I(y),  f\in\forall_{z\in A}\ (\ z\,\epsilon\, C(y,j) \ \rightarrow \ P(z)\ )]    \end{array}  }{\mathsf{ind}(m, q_1,q_2) \in P(a)     }}
  \\[20pt]
  \end{array}$$
$$\begin{array}{l}
\mbox{C$_1$-ind} 
\displaystyle{\frac{ 
    \begin{array}{l}    A\ set \ \ \ \  I(x)\  set \  [x\in A] \ \ \
      C(x,j)\ \in A\ \rightarrow\  \mathsf{prop_s}\ \ [x\in A, j\in I(x)]\\[5pt]
        P(x)\  prop\ [x\in A]\qquad  V\in A\ \rightarrow \mathsf{prop_s}\qquad \\[5pt]
      a\in A\qquad   r \in a \,\epsilon\, V\qquad \\[5pt]
q_1(x,z)\in P(x)\ [ x\in A, z\in x\,\epsilon\, V]\qquad \qquad \\[5pt]
    q_2(y,j, f) \in P(y)\  [ y\in A, j\in I(y),  f\in\forall_{z\in A}\ (\ z\,\epsilon\, C(y,j) \ \rightarrow \ P(z)\ )]    \end{array} 
       }{\mathsf{ind}(\mathsf{rf}(a,r) , q_1,q_2)=q_1(a, r) \in P(a)      }}
\\[20pt]
\mbox{C$_2$-ind} 
\displaystyle{\frac{ 
    \begin{array}{l}    A\ set \ \ \ \  I(x)\  set \  [x\in A] \ \ \
      C(x,j)\ \in A\ \rightarrow\  \mathsf{prop_s}\  \ [x\in A, j\in I(x)]\\[5pt]
         P(x)\  prop\ [x\in A]\qquad  V\in A\ \rightarrow \mathsf{prop_s}\qquad \\[5pt]
    a\in A\qquad i\in I(a)\qquad    r\in\forall_{x\in A}\ (\ x\,\epsilon\, C(a,i)\ \rightarrow\ x\triangleleft_{I,C} V\ )
    \\[5pt] 
    q_1(x,z)\in P(x)\ [ x\in A, z\in x\,\epsilon\, V]\qquad \qquad \\[5pt]
    q_2(y,j, f) \in P(y)\  [ y\in A, j\in I(y),  f\in\forall_{z\in A}\ (\ z\,\epsilon\, C(y,j) \ \rightarrow \ P(z)\ )]    \end{array} }
  {\mathsf{ind}(\mathsf{tr}(a,i, r) , q_1,q_2)=q_2 (\, a\, ,i\, ,\, \lambda z.\lambda w.\, \mathsf{ind}(\, \mathsf{Ap}(\mathsf{Ap}(r,z),w)\, ,q_1,q_2)\, ) \in P(a)      }}
  \end{array}
  $$
    \\

Note that the cover relation preserves extensional equality of subsets represented as small propositional functions thanks to the induction principle:
\begin{lemma}
      For any axiom set in \mttind\  on a set $A$ with
      $ I(x)\  set\  [x\in A]$ and $  C(x,j)\in A\ \rightarrow \mathsf{prop_s} \ \ [x\in A, j\in I(x)]$ and for any propositional functions
      $  V_1\in  A\ \rightarrow \mathsf{prop_s}$ and $  V_2\in  A\ \rightarrow \mathsf{prop_s}$, there exists a proof-term
      $$q\in V_1=_{ext} V_2\ \rightarrow\  a\triangleleft_{I,C} V_1=_{ext}  a
      \triangleleft_{I,C} V_2 $$
      where for any small propositional functions $W_1$ and $W_2$
on a set $A$ we use the following abbreviation
      $$W_1=_{ext}W_2 \, \equiv\, \forall_{x\in A}\ (\ W_1(x)\ \leftrightarrow\ W_2(x)\ )$$
    \end{lemma}

        Recall that the interpretation of \emtt\ in \mtt\ in \cite{m09}
        interprets a set $A$ as an {\it extensional quotient} defined in \mtt\
        as a set $A^J$  of \mtt, called {\it support}, equipped
        with an equivalence relation $=_{A^J}$ over $A^J$, as well as families of
        sets are interpreted as families of extensional sets preserving the equivalence
        relations in their telescopic contexts. Now, 
         lemma~\ref{eqcov} suggests that  we can interpret an inductive basic cover  on a set $A$ of \emttind\  
         within \mttind\  as an inductive cover of \mttind\ on  the support $A^J$ by enriching the interpretation of the axiom-set in \mttind\ with the equivalence
         relation $=^{A^J}$ in a similar way to definition~\ref{indquo} as follows:
         \begin{definition}
          For any axiom set in \mttind\   represented by a set $A$ with
      $ I(x)\  set\  [x\in A]$ and $  C(x,j)\in A\ \rightarrow \mathsf{prop_s} \ \ [x\in A, j\in I(x)]$ 
      and for any given equivalence relation
      $x=_Ay \in   \mathsf{prop_s}\ [x\in A,y\in A]$ turning $A$ into an extensional set 
      as well as  the family of set  $ I(x)\  set\  [x\in A]$ and propositional functions $  C(x,j)\in A\ \rightarrow \mathsf{prop_s} \ \ [x\in A, j\in I(x)]$
      into an extensional family of sets and extensional propositional functions preserving $=_{A}$ according to the definitions in \cite{m09},
         we define a new axiom set as follows
         $$\begin{array}{l}
 A^{=_A}\, \equiv\, A\qquad \qquad I^{=_A}(x)\, \equiv\, I(x) + (\Sigma y\in A)(  x=_A y)\qquad \qquad \mbox{ for } x\in A\\
\end{array}
$$
where $C^{=_A}(a,j)$ is the formalization of
$$C^{=_A}(a,j)\, \equiv\, \begin{cases}
   C(a,j)\, &  \mbox{ if } j\in  I(a)\\
  \{\, \pi_1(j) \, \} &  \mbox{ if } j\in (\Sigma y\in A)( a=_A y)\\
\end{cases}$$
for $a\in A$ and $j\in I^{=_A}(x)$.

\noindent We then call $\cov_{I,C}^{=_A}$ the inductive basic cover generated from this axiom set.
      \end{definition}

  We are now ready to interpret \emttind\ in the quotient model over \mttind\ built as in \cite{m09}:
    \begin{prop}\label{fullcv}
      The interpretation of \emtt\ in \mtt\ in \cite{m09} extends to an interpretation of \emttind\ in \mttind\ by interpreting an inductive basic cover $a \cov_{I,C} V$ for $a\in A$
      and $V\in {\mathcal P}(A)$
      as the inductive basic cover $\cov_{I^J,C^J}^{=_{A^J}}$  in \mttind\  over the support $A^J$ of the interpretation of $A$.
    \end{prop}
       \begin{proof}
       It follows from the proof given in \cite{m09} after checking that
       the inductive basic cover $\cov_{I^J,C^J}^{=_{A^J}}$  is an extensional proposition over the extensional set interpreting $A$ and over
      the interpretation of $ {\mathcal P}(A)$  in the sense of \cite{m09}.
       \end{proof}

 \section{The fragment \mluif\ of intensional Martin-L{\"o}f's type theory with inductive basic covers}
 We here briefly describe the theory \mluif\  obtained by adding
the rules of inductive basic covers
to the first order fragment of intensional Martin-L{\"o}f's type theory in  \cite{PMTT} with one universe.

 This theory  interprets \mttind\ as soon as propositions
 are identified with sets following the Curry-Howard correspondence
 in \cite{PMTT}.  In accordance with this propositions as sets interpretation, which is a pecularity of   Martin-L{\"o}f's type theory,
contrary to \mttind\  in \mluif\ we strengthen the elimination
 rule of inductive basic covers to act towards sets depending on their proof-terms according to inductive generation of types in Martin-L{\"o}f's type theory.

 As a consequence the interpretation of \mttind\ into \mluif\  also validates   the axiom of choice \ac\ as  formulated in the introduction.

 Therefore in order to show the consistency of \mttind\ with \ac+\ct\  (with
 \ct\ formulated as in the introduction) it is enough
 to show the consistency of \mluif\ extended with (the translation of) \ct.

 Here we adopt the notation of types and terms within the first order fragment \mlu\ of  intensional Martin-L{\"o}f's type theory  with one universe $U_0$  \`a la Tarski in \cite{IMMS} and we
 just describe the rule of inductive basic covers added to it. 

  To this purpose we add to \mlu\  the code
 $$ a\, \widehat{\triangleleft}_{s,i,c}\, v\in U_0\qquad \mbox{ for } a\in \mathsf{T}(s) \mbox{ and }
 v\in \mathsf{T}(s)\ \rightarrow \ U_0$$
 meaning that {\it the element $a$ of a small set  $\mathsf{T}(s)$  represented by the code $s\in U_0$ is covered by the subset
 $v$} represented by a small propositional function from $\mathsf{T}(s)$
 to the (large) set of small propositions identified with $U_0$ by the propositions-as-sets correspondence.

 Moreover,
  we use the abbreviations
$$a\triangleleft_{s,i,c} v\, \equiv\,  \mathsf{T}(a\,\widehat{\triangleleft}_{s,i,c}\, v)  \qquad \qquad x\,\epsilon\, y\, \equiv\, \mathsf{T}(\mathsf{Ap}(y,x))$$ 


 and  the notation
$$\mathsf{axcov}(s,i,c)$$
 to abbreviate the following judgements
$$s\in U_0 \ \ \ \  i(x)\in U_0 \  [x\in \mathsf{T}(s)] \ \ \
 c(x,y)\in \mathsf{T}(s)\rightarrow U_0 \ [x\in \mathsf{T}(s), y\in \mathsf{T}(i(x))]$$
 
  Then, the precise rules of inductive basic covers extending \mlu\ to form a new type system \mluif\ are the following:\\

 {\bf Rules of inductively generated basic covers in \mluif}

  $$\begin{array}{l}
\mbox{\rm F-$\triangleleft$} \
\displaystyle{\frac{\begin{array}{l}
    \mathsf{axcov}(s,i,c)\qquad 
      a\in \mathsf{T}(s)\qquad  v\in \mathsf{T}(s)\rightarrow U_0\ \end{array} }
  {\displaystyle a\,\widehat{\triangleleft}_{s,i,c}\, v\in U_0 }}\\
  \\[20pt]
  \end{array}$$
  $$\begin{array}{l}
\mbox{\rm rf-$\triangleleft$} \
\displaystyle{\frac{\begin{array}{l}
     \mathsf{axcov}(s,i,c)\qquad 
   a\in \mathsf{T}(s)\qquad  v\in \mathsf{T}(s)\rightarrow U_0\qquad\  r\in a\,\epsilon\, v\ \end{array} }
  {\displaystyle\mathsf{rf}(a,r)\in a \triangleleft_{s,i,c} v}}\\
\\[20pt]
\end{array}$$
$$\begin{array}{l}
\mbox{\rm tr-$\triangleleft$} \
\displaystyle{
  \frac{\begin{array}{l} \mathsf{axcov}(s,i,c)\qquad
      a\in \mathsf{T}(s)\qquad j\in \mathsf{T}(i(a))\qquad   v\in \mathsf{T}(s)\rightarrow U_0 \\
      r\in (\Pi z\in \mathsf{T}(s))(z\,\epsilon\,c(a,j)\rightarrow z \triangleleft_{s,i,c}v) \end{array} }
  {\displaystyle \mathsf{tr}(a,j,r) \in a\triangleleft_{s,i,c}v }}\\
\\[10pt]
\end{array}$$
      $$\begin{array}{l}
        \mbox{ind-$\triangleleft$} 
\displaystyle{\frac{ 
    \begin{array}{l}    \mathsf{axcov}(s,i,c)\\        v\in \mathsf{T}(s)\rightarrow U_0 \qquad P(x,u)\  type\ [x\in \mathsf{T}(s), u\in x\triangleleft_{s,i,c}v]\qquad \\
     a \in \mathsf{T}(s) \qquad m\in a\triangleleft_{s,i,c}v\\
          q_1 (x,w)\in P(x,\mathsf{rf}(x,w))\ [x\in \mathsf{T}(s), w\in  x\,\epsilon\, v]\\[3pt]
          q_2(x,h ,k,f)\in P(x,\mathsf{tr}(x,h,k))\ \ \\
          \quad \qquad  [x\in \mathsf{T}(s), h\in \mathsf{T}(i(x)),\\
            \qquad \qquad  k\in (\Pi z\in \mathsf{T}(s))(z\,\epsilon\,c(x,h)\rightarrow z \triangleleft_{s,i,c}v),\\
           \qquad \qquad\qquad f\in
            (\Pi z\in \mathsf{T}(s))(\Pi u\in z\,\epsilon\,c(x,h))\, P(z, \mathsf{Ap}(\mathsf{Ap}(k,z),u))]       
  \end{array} }{\mathsf{ind}(m, q_1,q_2)\in  P(a,m)     }}\\[20pt]
  \end{array}$$
$$  \begin{array}{l}
  \mbox{C$_1$-ind-$\triangleleft$} 
\displaystyle{\frac{ 
    \begin{array}{l}   \mathsf{axcov}(s,i,c)\\  
        v\in \mathsf{T}(s)\rightarrow U_0 \qquad P(x,u)\  type\ [x\in \mathsf{T}(s),u\in x\triangleleft_{s,i,c}v ]\\
     a \in \mathsf{T}(s) \qquad r\in a\,\epsilon\, v\\
          q_1 (x,w)\in P(x,\mathsf{rf}(x,w))\ [x\in \mathsf{T}(s), w\in  x\,\epsilon\, v]\\[5pt]
          q_2(x,h,k,f  )\in P(x,\mathsf{tr}(x,h,k))\ \ \\
        \quad \qquad  [x\in \mathsf{T}(s), h\in \mathsf{T}(i(x)),\\
            \qquad \qquad  k\in (\Pi z\in \mathsf{T}(s))(z\,\epsilon\,c(x,h)\rightarrow x \triangleleft_{s,i,c}v), \\
            \qquad \qquad\qquad f\in
            (\Pi z\in \mathsf{T}(s))(\Pi u\in z\,\epsilon\,c(x,h))\, P(z, \mathsf{Ap}(\mathsf{Ap}(k,z),u))]       
  \end{array} }{\mathsf{ind}(\mathsf{rf}(a,r), q_1,q_2)= q_1(a,r )\in  P(a,\mathsf{rf}(a,r))   }}\\[20pt]
    \end{array}$$
$$  \begin{array}{l}
\mbox{C$_2$-ind-$\triangleleft$} 
\displaystyle{\frac{ 
    \begin{array}{l}  \mathsf{axcov}(s,i,c)\\  
        v\in \mathsf{T}(s)\rightarrow U_0 \qquad P(x,u)\  type\ [x\in \mathsf{T}(s),u\in x\triangleleft_{s,i,c}v ]\\
      a \in \mathsf{T}(s) \qquad j\in \mathsf{T}(i(a))  \qquad
       r\in (\Pi z\in \mathsf{T}(s))(z\,\epsilon\,c(a,j)\rightarrow z \triangleleft_{s,i,c}v)\\
          q_1 (x,w)\in P(x,\mathsf{rf}(x,w))\ [x\in \mathsf{T}(s), w\in  x\,\epsilon\, v]\\[5pt]
          q_2(x,h ,k,f  )\in P(x,\mathsf{tr}(x,h,k))\ \ \\
        \quad \qquad  [x\in \mathsf{T}(s), h\in \mathsf{T}(i(x)),\\
            \qquad \qquad  k\in (\Pi z\in \mathsf{T}(s))(z\,\epsilon\,c(x,h)\rightarrow z \triangleleft_{s,i,c}v),\\
            \qquad \qquad\qquad f\in
            (\Pi z\in \mathsf{T}(s))(\Pi u\in z\,\epsilon\,c(x,h))P(z, \mathsf{Ap}(\mathsf{Ap}(k,z),u))]         \end{array} }{\mathsf{ind}(\mathsf{tr}(a,j,r), q_1,q_2)= q_2(a,j,r,\lambda z.\lambda u.\mathsf{ind}(\mathsf{Ap}(\mathsf{Ap}(r,z),u),q_1,q_2))\in  P(a,\mathsf{tr}(a,j,r))     }}\\[20pt]
\end{array}$$

A crucial difference from the ordinary versions of Martin-L\"of's type theory
is that for $\mluif$ we {\it postulate just the replacement rule repl)

\noindent
$$
\begin{array}{l}
      \mbox{repl)} \ \
\displaystyle{ \frac
         { \displaystyle 
\begin{array}{l}
 c(x_1,\dots, x_n)\in C(x_1,\dots,x_n)\ \
 [\, x_1\in A_1,\,  \dots,\,  x_n\in A_n(x_1,\dots,x_{n-1})\, ]   \\[2pt]
a_1=b_1\in A_1\ \dots \ a_n=b_n\in A_n(a_1,\dots,a_{n-1})
\end{array}}
         {\displaystyle c(a_1,\dots,a_n)=c(b_1,\dots, b_n)\in
 C(a_1,\dots,a_{n})  }}
\end{array}     
$$
  in place of the usual congruence rules which would include the $\xi$-rule} in accordance
with the rules of \mtt\ in \cite{m09}, and hence of \mttind.

The motivation for this restriction in \mttind\ and in \mluif\  is due to the fact that
the realizability semantics we present in the next sections, based on that
in \cite{IMMS} and hence
on the original Kleene realizability in \cite{DT88}, does not validate 
the $\xi$-rule\footnote{Notice that a trivial instance of the $\xi$-rule is derivable from repl) when $c$ and $c'$ don't depend on $x^B$.} of lambda-terms
$$
\mbox{  $\xi$} \
\displaystyle{\frac{ \displaystyle c=c'\in  C\ [x\in B]  }
{ \displaystyle \lambda x^{B}.c=\lambda x^{B}.c' \in (\Pi x\in B) C}}
$$
which is instead valid in \cite{PMTT}.

It is indeed an open problem whether
the original intensional version of  Martin-L{\"o}f's type theory  in \cite{PMTT}, including the
$\xi$-rule of lambda terms, is consistent with \ct.

It is worth noting that  the lack of the $\xi$-rule  does not affect the possibility of adopting \mtt\ as
the intensional
level of a  two-level constructive foundation
as intended in  \cite{mtt}, since 
  its term equality rules suffice to interpret
an extensional level including extensionality of functions, as that represented
by \emtt, by means of the quotient model as introduced in \cite{m09} and studied abstractly
in ~\cite{elqu,qu12,uxc}.

Furthermore our realizability semantics interprets terms
as applicative terms in the first Kleene algebra and their
equality as numerical
equality turning into an extensional equality in the context-dependent case.
Hence we need a suitable
encoding of lambda-terms which validates the replacement rule under the interpretation.  As observed in \cite{IMMS} not each translation of pure lambda calculus
in the first Kleene algebra satisfies this requirement (see pp.881-882 in \cite{IMMS}).

Now note that we can interpret \mttind\ within \mluif\ by first interpreting \mttind\ in the version of \mluif\ with the universe $U_0$ \`a la Russell
which is then interpreted into the original version \mluif\ with the first universe \`a la Tarski. 

 In more detail, we call \mlur\ the first order fragment of intensional Martin-L{\"o}f's type theory in  \cite{PMTT} with the first universe $U_0$ formulated \`a la Russell.  Then, we call \mluifr\ 
 the version of \mluif\ with the universe $U_0$ and the rules of inductively generated basic covers \`a la Russell which are obtained from those of \mluif\ 
  by identifying the code of  a set in the first universe $U_0$
with the set itself.
In particular in \mluifr\ we must have a new set in $U_0$ expressing  the basic cover 
$$ a\, \triangleleft_{I,C} V \qquad \mbox{ for  a set $A$ in $U_0$ } \mbox{ and }
 V\in A\ \rightarrow \ U_0$$
 for any  axiom set given by a set family $I(a)$   in  $U_0$ for  $a\in A$  and $C(a,j)\in  A\ \rightarrow \ U_0 $ for $a\in A$ and $j\in I(a)$.

\begin{theorem}\label{first}
  The interpretation of \mtt\ into \mlur\  given in \cite{m09}
  extends to that of \mttind\   by interpreting 
  each basic cover $\cov_{I,C}$ of \mttind\ associated to an axiom set $I(-)$ and $C(-,-)$  in the corresponding
  basic cover of \mluifr\ associated to the interpreted axiom set.
  \end{theorem}
\begin{proof}
  Note that  small propositions  are encoded in the universe $U_{0}$
    as well as axiom sets generating a basic cover  in \mttind.
\end{proof}

Observe that the version \`a la Russell \mluifr\ can be interpreted in that \`a la Tarski \mluif\  by preserving the meaning of sets and of their elements:
\begin{prop}\label{second}
   \mluifr\ can be interpreted  into \mluif\  in such a way that any set $A$ (under a context)  in the first universe is interpreted as a set $\mathsf{T}(c)$ for some code $c$  and
  each basic cover $a\cov_{I,C}V$ of \mluifr\ for $a\in A$ and $V\in A\rightarrow U_0$ associated to an axiom set $I(-)$ and $C(-,-)$  is interpreted as the set with code
 $  a\, \widehat{\triangleleft}_{s,i,c}\, v\in U_0  \mbox{ for } a\in \mathsf{T}(s) \mbox{ and } v\in \mathsf{T}(s)\ \rightarrow \ U_0$ where $s$ is the code of the set interpreting $A$, $v$ is the interpretation of $V$ and $i(x)\in U_0 \  [x\in \mathsf{T}(s)] $ and
 $c(x,y)\in \mathsf{T}(s)\rightarrow U_0 \ [x\in \mathsf{T}(s), y\in \mathsf{T}(i(x))]$
 are the interpretation of the axiom-set.
  \end{prop}
\begin{proof}
The raw syntax of \mluif, i.e.~the preterms and pretypes associated to the syntax of terms
and types of \mluif, is defined in the usual way (see \cite{IMMS}). Observe that a set in the first universe \`a la Russell may appear
both as a preterm and as a pretype.

Then we define a partial interpretation of preterms and pretypes respectively as preterms and pretypes of \mluif\ with the warning of interpreting a set in the first universe $U_0$ used as a preterm
as its corresponding code \`a la Tarski in \mluif.
\end{proof}

\begin{cor}\label{intttmtt}
   \mttind\  can be interpreted in \mluif\ by composing the interpretations in  theorem~\ref{first} and proposition~\ref{second}.  \end{cor}

\section{A realizability interpretation of \mluif\ with  the formal Church's Thesis}
Here we are going to describe 
a realizability model of \mluif\ with \ct\
extending that of \mlu\ in \cite{IMMS}.

A main novelty here is that we formalize such a model
in the (generalized) predicative and constructive theory $\mathbf{CZF+REA}$ where
$\mathbf{CZF}$ stands for Constructive Zermelo-Fraenkel Set Theory and $\mathbf{REA}$ stands for the regular extension axiom (for details see \cite{Aczel3,czf,czf2}).

Since the interpretation in \cite{IMMS} is performed in 
$\tar$ which  is a classical theory of fixed points,
we cannot follow the proof technique in \cite{IMMS} to fulfill our
purpose. Moreover $\tar$ is a too weak theory to accommodate inductively defined topologies as it can be gleaned from \cite{CR}. The solution is to adopt the proof-technique in \cite{R93, RG94} to fulfill our goal.

 As usual in set theory we identify the natural numbers with the finite ordinals, i.e.\ $\mathbb{N}:=\omega$. To simplify the treatment we will assume that $\mathbf{CZF}$ has names for all (meta) natural numbers. Let $\overline{n}$ be the constant designating the $n^{th}$ natural number. We also assume that $\mathbf{CZF}$ has function symbols for addition and multiplication on $\mathbb{N}$ as well as for a primitive recursive bijective pairing function $\pp:\mathbb{N}\times \mathbb{N}\rightarrow \mathbb{N}$ and its primitive recursive inverses $\pp_0$ and $\pp_1$, that satisfy $\pp_0(\pp(n,m))=n$ and $\pp_1(\pp(n,m))=m$. We also assume that $\mathbf{CZF}$ is endowed with symbols for a primitive recursive length function $\ell:\mathbb{N}\rightarrow \mathbb{N}$ and a primitive recursive component function $(-)_{-}:\mathbb{N}\times \mathbb{N}\rightarrow \mathbb{N}$ determining a bijective encoding of finite lists of natural numbers by means of natural numbers. $\mathbf{CZF}$ should also have a symbol $T$ for Kleene's $T$-predicate and the result extracting function $U$. Let $P(\{e\}(n))$ be a shorthand for $\exists m(T(e,n,m)\wedge P(U(m)))$. Further, let $\pp(n,m,k):=\pp(\pp(n,m),k)$, $\pp(n,m,k,h):=\pp(\pp(n,m,k),h)$, etc. 
 A similar convention will be adopted for application of partial recursive functions: Let $\{e\}(a,b):=\{\{e\}(a)\}(b)$,  $\{e\}(a,b,c):=\{\{e\}(a,b)\}(c)$ etc.
 We use $a,b,c,d,e,f,n,m,l,k,q,r,s,v,j,i$ as metavariables for natural numbers.  

We first need to introduce some abbreviations: 
\begin{enumerate}
\item $\mathsf{n}_0$ is $\pp(0,0)$, $\mathsf{n}_1$ is $\pp(0,1)$ and $\mathsf{n}$ is $\pp(0,2)$.
\item $\sigmat(a,b)$ is $\pp(1,\pp(a,b))$, $\pit(a,b)$ is $\pp(2,\pp(a,b))$ and $+(a,b)$ is $\pp(3,\pp(a,b))$.
\item $\mathsf{list}(a)$ is $\pp(4,a)$ and $\mathsf{id}(a,b,c)$ is $\pp(5,\pp(a,b,c))$.
\item $a\widetilde{\triangleleft}_{c,d,e}b$ is $\pp(6,\pp(a,b,c,d,e))$.
\item $\rft(a,r)$ is  $\pp(7,\pp(a,r))$.
\item $\trt(a,j,r)$ is $\pp(8,\pp(a,j,r))$. 
\end{enumerate}

Recall that in intuitionistic set theories ordinals are defined as transitive sets all of whose members are transitive sets, too. Unlike in the classical case, one cannot prove that they are linearly ordered but they are perfectly good as a scale along which one can iterate various processes. The trichotomy of $0$, successor, and limit ordinal, of course, has to be jettisoned. We use lowercase greek letters as metavariables for ordinals.

\begin{definition}\label{fixdef} By transfinite recursion on ordinals (cf. \cite{czf2}, Proposition 9.4.4) we define simultaneously two relations $\mathsf{Set}_\alpha(n)$ and $n\,\varepsilon _\alpha\,m$ on $\mathbb{N}$ in $\mathbf{CZF+REA}$.

In the following definition we use the shorthand $\mathsf{Fam}_\alpha(e,k)$ to convey that $\mathsf{Set}_\alpha(k)$ and $\forall j(j\,\varepsilon _\alpha\,k\rightarrow \mathsf{Set}_\alpha(\{e\}(j)))$ and we shall write $\mathsf{Set}_{\in \alpha}(n)$ for $\exists \beta\in \alpha(\mathsf{Set}_\beta (n))$, $n\,\varepsilon _{\in\alpha}\,m$ for $\exists \beta\in \alpha(n\,\varepsilon _\beta\,m)$ and $\mathsf{Fam}_{\in\alpha}(e,k)$ for $\exists \beta\in \alpha(\mathsf{Fam}_\beta(e,k))$.
\begin{enumerate}[itemsep = 0.5em]
\item $\mathsf{Set}_\alpha(\mathsf{n}_j)$ if $j=0$ or $j=1$, and $m\,\varepsilon _\alpha\,\mathsf{n}_j$ if $m<j$;
\item $\mathsf{Set}_\alpha(\mathsf{n})$ holds, and $m\,\varepsilon _\alpha\,\mathsf{n}$ if $m\in \mathbb{N}$.
\item If $\mathsf{Fam}_{\in \alpha}(e,k)$, then $\mathsf{Set}_\alpha(\pit(k,e))$ and $\mathsf{Set}_\alpha(\sigmat(k,e))$;
\\ if $\mathsf{Fam}_{\in \alpha}(e,k)$, then
\begin{enumerate}
\item $n\,\varepsilon _\alpha\,\pit(k,e)$ if there exists $\beta\in \alpha$ such that $\mathsf{Fam}_{\beta}(e,k)$ and 
$(\forall i\,\varepsilon _\beta\,k)\, \{n\}(i)\,\varepsilon_\beta\, \{e\}(i)$.\footnote{We use the obvious shorthand  $(\forall i\,\varepsilon _\beta\,k) \ldots$ for 
$\forall i[i\,\varepsilon_\beta\,k\rightarrow\ldots]$; also employed henceforth.} 
\item $n\,\varepsilon_\alpha\,\sigmat(k,e)$ if there exists $\beta\in \alpha$ such that $\mathsf{Fam}_{\beta}(e,k)$ as well as $\pp_0(n)\,\varepsilon_\beta\,k$ and $\pp_1(n)\,\varepsilon_\beta\, \{e\}(\pp_0(n))$.
\end{enumerate}

\item If there exists $\beta\in\alpha$ such that $\mathsf{Set}_\beta(n)$ and $\mathsf{Set}_\beta(m)$, then $\mathsf{Set}_\alpha(+(n,m))$, and 
\\ $i\,\varepsilon _\alpha {+}(n,m)$ if there exists $\beta\in\alpha$ such that $\mathsf{Set}_\beta(n)$, $\mathsf{Set}_\beta(m)$ and 
$$[\pp_0(i)=0\wedge \pp_1(i)\,\varepsilon_\beta\,n]\,\vee\,[\pp_0(i)=1\wedge \pp_1(i)\,\varepsilon_\beta\,m].$$

\item  If there exists $\beta\in\alpha$ such that $\mathsf{Set}_\beta(n)$, then $\mathsf{Set}_\alpha(\mathsf{list}(n))$, and 
\\ $i\,\varepsilon_\alpha\, \mathsf{list}(n)$ if there exists $\beta\in\alpha$ such that $\mathsf{Set}_\beta(n)$ and $\forall j\,[j<\ell(i)\rightarrow (i)_j \,\varepsilon_\beta\,n]$.

\item If $\mathsf{Set}_{\in \alpha}(n)$, then $\mathsf{Set}_\alpha(\mathsf{id}(n,m,k))$, and 
\\ $s\,\varepsilon_\alpha\,\mathsf{id}(n,m,k)$ if there exists $\beta\in \alpha$ such that $\mathsf{Set}_{\beta}(n)$, $m\,\varepsilon_\beta\,n$ and $s=m=k$.

\item Let $\beta\in \alpha$. Suppose that the following conditions (collectively called $*_\beta$) are satisfied: 
\begin{enumerate}
\item $\mathsf{Set}_\beta(s)$, 
\item $a\,\varepsilon _\beta\,s$, 
\item $\mathsf{Fam}_\beta(v,s)$,  
\item $\mathsf{Fam}_\beta(i,s)$ and
\item $\forall x\forall y[x\,\varepsilon _\beta s\,\wedge\, y\,\varepsilon _\beta \{i\}(x)\rightarrow \mathsf{Fam}_\beta( \{c\}(x,y),s)]$. 
\end{enumerate}
 Then $\mathsf{Set}_\alpha(a\widetilde{\triangleleft}_{s,i,c} v)$.
\\ For $\beta\in\alpha$ satisfying $*_\beta$, let $V_{\beta}$ be the smallest subset of $\mathbb{N}$ satisfying the following conditions:
\begin{enumerate}
\item if $z\, \varepsilon _\beta s$ and $r \,\varepsilon _\beta  \{v\}(z)$ then $\pp(z,\rft(z,r))\in V_{\beta}$;
\item if $r\in\mathbb{N}$, $z\, \varepsilon_\beta\, s $,  $j \,\varepsilon_\beta  \{i\}(z)$ 
and $$(\forall u\, \varepsilon_\beta s)\,(\forall t\,\varepsilon_\beta \{c\}(z,j,u))\ \pp(u,\{r\}(u,t))\in V_\beta$$
then $\pp(z,\trt(z,j,r))\in V_\beta$. 
\end{enumerate}
 The existence of the set $V_\beta$ is guaranteed by the axiom $\mathbf{REA}$. 
 \\ Finally we define $q\,\varepsilon _\alpha\,a\widetilde{\triangleleft}_{s,i,c} v$ by $\exists \beta\in \alpha\,[\,*_\beta \,\wedge\, \pp(a,q)\in V_{\beta}]$.

\end{enumerate}

\end{definition}

 \begin{remark}
   It is perhaps worth noting that in the above definition the interpretation of the Propositional Identity $T(\, \widehat{\mathsf{Id}}(s, a,b)\, )$ for $s\in U_0$ and $a\in \mathsf{T}(s)$ and
   $b\in \mathsf{T}(s)$ agrees with that in \cite{IMMS} which validates
   the rules of the {\it extensional Propositional Identity} in \cite{PMTT}. Hence,  also our realizability semantics actually validates the extensional version of
   \mluif.  Therefore the elimination rule of inductive basic covers
   can be equivalently weakened to act towards
   types which do not depend on  proof-terms of basic covers, as soon as we add a suitable $\eta$-rule in a way similar to what happens to the rules of first-order
   types (like disjoint sums or natural numbers or list types) in the extensional type theories in \cite{tumscs}. 
   \end{remark}

Here we have a crucial lemma.

\begin{lem}[$\mathbf{CZF}+\mathbf{REA}$]\label{fix} \leavevmode
\begin{itemize} 
\item For all $m\in \mathbb{N}$, if $\mathsf{Set}_\alpha(m)$ and $\alpha\subseteq \rho$, then $\mathsf{Set}_\rho(m)$.
  \item
For all $m\in \mathbb{N}$, if $\mathsf{Set}_\alpha(m)$, then for all $\rho$ such that $\mathsf{Set}_\rho(m)$, 
\[\forall i\in \mathbb{N}(i\,\varepsilon _\alpha\,m\leftrightarrow i\,\varepsilon _\rho\,m).\tag{$*$}\]
\end{itemize}

\end{lem}
\begin{proof} The proof is similar to the one for Lemma 4.2 in \cite{R93,RG94}. (i) and (ii) are proved simultaneously  by induction on $\alpha$ (cf. \cite{czf2}, Proposition 9.4.3). Suppose $\mathsf{Set}_\alpha(m)$ and $\mathsf{Set}_\rho(m)$. We look at the forms $m$ can have.

If $m$ is $\mathsf{n}_0$, $\mathsf{n}_1$ or $\mathsf{n}$, then the claim is immediate in view of clauses $(1)$ and $(2)$ in the previous definition.

If $m$ is of the form $\pit(k,e)$, then there exists $\beta\in \alpha$ such that $\mathsf{Fam}_\beta(e,k)$. The induction hypothesis applied to $\beta$ yields that whenever $\mathsf{Fam}_\xi(e,k)$, then for all $\xi$
$$\forall j\in \mathbb{N}(i\,\varepsilon _\beta\,m\leftrightarrow i\,\varepsilon _\xi\,m)$$
$$\forall i\in \mathbb{N}\forall j\in \mathbb{N}(i\,\varepsilon _\beta\,m\rightarrow (j\,\varepsilon _\beta\,\{e\}(i)\leftrightarrow j\,\varepsilon _\xi\,\{e\}(i)))$$
The claim $(*)$  follows from the foregoing.  If $m$ is either $\sigmat(k,e)$, $+(a,b)$, $\mathsf{list}(a)$ or $\mathsf{id}(a,b,c)$ the argument proceeds as in the previous case.

If $m$ is of the form $a\widetilde{\triangleleft}_{s,i,c}v$, the proof is similar, although more involved.
\end{proof}

\begin{definition}We define in $\mathbf{CZF}+\mathbf{REA}$ the formula $\mathsf{Set}(n)$ as $\exists \alpha(\mathsf{Set}_\alpha(n))$ and $x\,\varepsilon \,y$ as $\exists \alpha(x\,\varepsilon _\alpha\, y)$.

\end{definition}

\begin{theorem}\label{real}
        Consistency of the theory $\mathbf{CZF}+\mathbf{REA}$  implies the consistency of the  theory \mluif\  extended with the formal Church thesis \ct.
        \end{theorem}
 \begin{proof} We outline a realizability semantics in {\bf CZF}+$\mathbf{REA}$. Let $\mathbf{p}$ be a code for the primitive recursive pairing function $\pp$ introduced just before Definition \ref{fixdef}, i.e.
$\{\mathbf{p}\}(n,m) =\pp(n,m)$.\footnote{Recall that when we write $\{b\}(a_{1},\ldots,a_{n})$, we mean $\{\ldots\{\{b\}(a_{1})\}(a_{2})\ldots\}(a_{n})$.}

Every preterm is interpreted as a $\mathcal{K}_1$-applicative term (that is, a term built with numerals, variables and Kleene application) as it is done in \cite{IMMS} with the only difference that here we do not identify $\mathsf{N}$ with $\mathsf{List}(\mathsf{N}_1)$ but we consider it as a primitive type interpreted as the set of natural numbers $\mathbb{N}$;  zero and successor terms  are interpreted in the obvious way while the eliminator relative to $\mathsf{N}$ is interpreted using a numeral encoding a recursor.

We must notice that in introducing codes for sets in the universe in Definition \ref{fixdef} we took account of dependencies by means of natural numbers representing recursive functions; however every preterm depending  on variables will be interpreted as a $\mathcal{K}_1$-applicative term having the same free variables (we identify the variables of \mluif\ with those in {\bf CZF}+$\mathbf{REA}$). For these reasons, whenever a term $s$ in \mluif\ depends on terms $t_1,\ldots,t_n$ in context, its interpretation will depend on the interpretations of $t_1,\ldots,t_n$ bounded with adequate $\Lambda$ operators. The variables which will be bounded by these $\Lambda$s will be the ones used in the rule where the term $s$ is introduced. This abuse of notation allows us to avoid heavy fully-annotated terms in the syntax. 

We only need to interpret the new preterms of \mluif\ as follows.
\begin{enumerate}

\item $(a\widehat{\triangleleft}_{s,i,c}v)^{I}$ is defined as $\{\mathbf{p}\}(6,\{\mathbf{p}_5\}(a^I,v^I,s^I,\Lambda x.i^I,\Lambda x.\Lambda y.c^I))$
($=a^I\;\widetilde{\triangleleft}_{s^I,\Lambda x.i^I,\Lambda x.\Lambda y.c^I}\;v^I$),\footnote{$(=\ldots)$ is meant to convey that the preceding term evaluates to same number as the one indicated by $\ldots$.}
 where 
$\mathbf{p}$ and $\mathbf{p}_5$ are numerals representing the encoding of pairs of natural numbers and of $5$-tuples of natural numbers, respectively.
\item $(\mathsf{rf}(a,r))^I:=\{\mathbf{p}\}(7,\{\mathbf{p}\}(a^I,r^I))$ ($=\rft(a^I,r^I)$).
\item $(\mathsf{tr}(a,j,r))^I:=\{\mathbf{p}\}(8,\{\mathbf{p}_3\}(a^I,j^I,r^I))$  ($=\trt(a^I,j^I,r^I)$), where $\mathbf{p}_3$ is a numeral representing the encoding of triples of natural numbers.

\item $(\mathsf{ind}(m,q_1,q_2))^I$ is $\{\mathbf{ind}_{q_1,q_2}\}(m^I)$ where $\mathbf{ind}_{q_1,q_2}$ is the code of a recursive function\footnote{Actually depending primitive recursively on the parameters, i.e. free variables, occuring in $q_1$ and $q_2$.}  such that 
\begin{enumerate}
\item $\mathbf{ind}_{q_1,q_2}(\rft(z,r))\simeq\{\Lambda x.\Lambda w. q_1^I\}(z,r)~\footnote{ $t\simeq s$ to mean that the applicative term $t$ converges if and only if  $s$ converges, and in this case $t=s$. }$,
\item $\mathbf{ind}_{q_1,q_2}(\trt(a,i,r))\simeq\{\Lambda x.\Lambda h.\Lambda k.\Lambda f.q_{2}^I\}(a,i,r,\Lambda z.\Lambda u.\mathbf{ind}_{q_1,q_2}(\ \{r\}(z,u)) )$.
\end{enumerate}
For the existence of such a  code one  appeals to the recursion theorem.
\end{enumerate}

\noindent 
If $\tau$ is a $\mathcal{K}_1$-applicative term and $A=\{x|\,\phi\}$ is a class, we will define $\tau\,\in\, A$ as an abbreviation for $\phi[\tau/x]$.

\noindent We will interpret pretypes into the language of set theory (with $\mathbf{CZF}+\mathbf{REA}$ being the interpreting theory) as definable subclasses of $\mathbb{N}$ 
 as follows.
\begin{enumerate}[align=left]
\item $\mathsf{N}_{0}^{I}:=\{x\in \mathbb{N}|\,\bot\}$.
\item $\mathsf{N}_{1}^{I}:=\{x\in \mathbb{N}|\,x=0\}$.
\item $\mathsf{N}^{I}:=\{x\in \mathbb{N}|\,x=x\}=\mathbb{N}$.
\item $((\Sigma y\in A)B)^{I}:=\{x\in \mathbb{N}|\,\pp_0(x)\in A^{I}\wedge \pp_1(x)\in B^{I}[\pp_0(x)/y] \}$.
\item $((\Pi y\in A)B)^{I}:=\{x\in \mathbb{N}|\,\forall y\in \mathbb{N}\,[y\in A^{I}\rightarrow \{x\}(y)\in B^{I}]\}$.
\item $(A+B)^{I}:=\{x\in \mathbb{N}|\,[\pp_0(x)=0 \wedge \pp_1(x)\in A^{I}]\,\vee\,[\pp_0(x)=1 \wedge \pp_1(x)\in B^{I}]\}$.
\item $(\mathsf{List}(A))^{I}:=\{x\in \mathbb{N}|\,\forall i\in \mathbb{N}\,[i<\ell(x)\rightarrow (x)_{i}\in A^{I}]\}$.
\item $(\mathsf{Id}(A,a,b))^{I}:=\{x\in \mathbb{N}|\,x=a^{I}\wedge a^{I}=b^{I}\, \wedge\,  a^{I}\in A^{I}\}$.
\item $U_0^{I}:=\{x|\,\mathsf{Set}(x)\}$.
\item $\mathsf{T}(a)^{I}:=\{x|\,x\,\varepsilon \,a^{I}\}$.
\end{enumerate}

\noindent Precontexts are interpreted as conjunctions of set-theoretic formulas  as follows. 
\begin{enumerate}
\item $[\;]^{I}$ is the formula $\top$;
\item $[\Gamma, x\in A]^{I}$ is the formula $\Gamma^{I}\,\wedge\, x\in A^{I}$.
\end{enumerate}

\noindent Validity of judgements $J$ in $\mathbf{CZF}+\mathbf{REA}$ under the foregoing interpretation is defined as follows:

\begin{enumerate}
\item $A\, type\,[\Gamma]$ holds if $\mathbf{CZF}+\mathbf{REA}\vdash \Gamma^{I}\to\forall x\,(x\in A^I\rightarrow x\in \mathbb{N})$ 
\item $A=B\, type\,[\Gamma]$ holds if  $\mathbf{CZF}+\mathbf{REA}\vdash \Gamma^{I}\to\forall x\,(x\in A^{I}\leftrightarrow x\in B^{I})$
\item $a\in A\,\,[\Gamma]$ holds if  $\mathbf{CZF}+\mathbf{REA}\vdash \Gamma^{I}\to a^{I}\in A^{I}$
\item $a=b\in A\, \,[\Gamma]$ holds if  $\mathbf{CZF}+\mathbf{REA}\vdash \Gamma^{I}\to a^{I}\in A^{I}\wedge a^{I}=b^{I}$,
\end{enumerate}
where $x$ is a fresh variable.

The encoding of lambda-abstraction in terms of $\mathcal{K}_1$-applicative terms can be chosen (see \cite{IMMS}) in such a way that if $a$ and $b$ are terms and $x$ is a variable which is not bounded in $a$, then the terms $(\,a[b/x]\,)^{I}$ and $a^{I}[\,b^{I}/x]$ coincide.

The proof that for every judgement if $\mluif\vdash J$, then $J$ holds in the realizability model is a long verification. 
Serving as a generic example, we will prove 
that the rules for the inductively generated covers (rf-$\triangleleft$) and (tr-$\triangleleft$) preserve the validity of judgments in the model.
In doing so we will tacitly be assuming that the interpretation of the  context $\Gamma^{I}$  holds true.

\begin{enumerate}[align=left]
\item[(rf-$\triangleleft$)]  
Suppose the premisses of the following rule are valid in the model.
$$\mbox{\rm rf-$\triangleleft$} \
\displaystyle{\frac{\begin{array}{l}
      s \in U_0 \qquad i(x)\in U_0\  [x\in \mathsf{T}(s)] \ \ \ c(x,y)\in \mathsf{T}(s)\rightarrow U_0 \ \ [x\in \mathsf{T}(s), y\in \mathsf{T}(i(x))]\\
   a\in \mathsf{T}(s)\qquad  v\in \mathsf{T}(s)\rightarrow U_0\qquad\qquad  r\in a\,\epsilon\, v\ \end{array} }
  {\displaystyle\mathsf{rf}(a,r)\in a \triangleleft_{s,i,c} v}}$$
Then, the following  hold true: 
\begin{enumerate}[align=left]
\item $\mathsf{Set}(s^I)$
\item $\forall x\in \mathbb{N}\,(x\,\varepsilon\, s^I\rightarrow \mathsf{Set}((i(x))^I))$
\item $\forall x,y,z\in \mathbb{N}\,(x\,\varepsilon\, s^I\wedge y \,\varepsilon\, (i(x))^I\wedge x\,\varepsilon\, s^I\rightarrow  \mathsf{Set}(\{(c(x,y))^I\}(z)))$
\item $a^I\,\varepsilon \, s^I$
\item $\forall x\in \mathbb{N}\,(x\,\varepsilon\, s^I\rightarrow \mathsf{Set}(\{v^I\}(x)))$
\item $r^I\,\varepsilon \, \{v^I\}(a^I)$
\end{enumerate}
(a) entails that there exists an ordinal $\alpha$ such that \begin{eqnarray}\label{MRG0}&&\mathsf{Set}_{\alpha}(s^I)\end{eqnarray} and by (b) we have
 $\forall x\in \mathbb{N}\,(x\,\varepsilon_{\alpha}\, s^I\rightarrow \exists\beta\,\mathsf{Set}_{\beta}((i(x))^I))$. Using strong collection in our background theory, the latter yields the existence of a set of ordinals $b_1$ such that 
  \begin{eqnarray}\label{MRG1} &&\forall x\in \mathbb{N}\,(x\,\varepsilon_{\alpha}\, s^I\rightarrow \exists\beta\in b_1\,\mathsf{Set}_{\beta}((i(x))^I)).\end{eqnarray}
  Let $\gamma$ be $\alpha\cup b_1\cup \bigcup b_1$. One then determines that $\gamma$ is an ordinal. From (c) it follows by strong collection that there exists a set of ordinals $b_2$
  such that 
\begin{eqnarray}\label{MRG2} && \forall x,y,z\in \mathbb{N}\,(x\,\varepsilon_{\gamma}\, s^I\wedge y \,\varepsilon_{\gamma}\, (i(x))^I\wedge x\,\varepsilon_{\gamma}\, s^I\rightarrow  \exists \xi\in b_2\,\mathsf{Set}_{\xi}(\{(c(x,y))^I\}(z))).\end{eqnarray}
Likewise, (e) yields the existence of a set of ordinals $b_3$ such that 
\begin{eqnarray}\label{MRG3} &&\forall x\in \mathbb{N}\,(x\,\varepsilon_{\gamma}\, s^I\rightarrow \exists\zeta\in b_3\,\mathsf{Set}_{\zeta}(\{v^I\}(x))).\end{eqnarray}
  Finally, let $\delta$ be the ordinal $\gamma \cup b_2 \cup  \bigcup b_2 \cup b_3\cup \bigcup b_3$. In light of Lemma \ref{fix}, we can infer from  the above (\ref{MRG0}, \ref{MRG1}, \ref{MRG2} and \ref{MRG3})  collectively  that 
 \begin{itemize}[align=left]
\item[(i)] $\mathsf{Set}_{\delta}(s^I)$
\item[(ii)] $\forall x\in \mathbb{N}\,(x\,\varepsilon_{\delta}\, s^I\rightarrow \mathsf{Set}_{\delta}((i(x))^I))$
\item[(iii)] $\forall x,y,z\in \mathbb{N}\,(x\,\varepsilon_{\delta}\, s^I\wedge y \,\varepsilon_{\delta}\, (i(x))^I\wedge x\,\varepsilon_{\delta}\, s^I\rightarrow  \mathsf{Set}_{\delta}(\{(c(x,y))^I\}(z)))$
\item[(iv)] $a^I\,\varepsilon_{\delta} \, s^I$
\item[(v)] $\forall x\in \mathbb{N}\,(x\,\varepsilon_{\delta}\, s^I\rightarrow \mathsf{Set}_{\delta}(\{v^I\}(x)))$
\item[(vi)] $r^I\,\varepsilon_{\delta} \, \{v^I\}(a^I)$
\end{itemize}
Let $\theta$ be any ordinal such that $\delta\in \theta$. By Definition \ref{fixdef}(7), we then have
 $\mathsf{Set}_\theta(a^I\;\widetilde{\triangleleft}_{s^I,i^I,c^I}\; v^I)$ and  $\mathsf{rf}(a,r)^I=\rft(a^I,r^I)\,\varepsilon \,a^I\;\widetilde{\triangleleft}_{s^I,i^I,c^I}\,v^I$.
  So the validity of the judgement $\mathsf{rf}(a,r)\in a \triangleleft_{s,i,c} v$ in the model is established.

\item[(tr-$\triangleleft$)] Suppose the premisses of the following rule are valid in the model.
$$\mbox{\rm tr-$\triangleleft$} 
\displaystyle{
  \frac{\begin{array}{l} s\in U_0\ \ \ \  i(x)\in U_0 \  [x\in \mathsf{T}(s)] \ \ \ c(x,y)\in \mathsf{T}(s)\rightarrow U_0 \ \ [x\in \mathsf{T}(s), y\in \mathsf{T}(i(x))]\\
      a\in \mathsf{T}(s)\qquad j\in \mathsf{T}(i(a))\qquad   v\in \mathsf{T}(s)\rightarrow U_0 \\
      r\in (\Pi x\in \mathsf{T}(s))(x\,\epsilon\,c(a,j)\rightarrow x \triangleleft_{s,i,c}v) \end{array} }
  {\displaystyle \mathsf{tr}(a,j,r) \in a\triangleleft_{s,i,c}v }}$$
Then, in addition to  (a)-(e) of the previous case, $j^{I}\,\varepsilon \,\{i^{I}\}(a^{I})$ and $$\forall x\in \mathbb{N}\,\forall y\in \mathbb{N}\,[x\,\varepsilon \,s^{I}\wedge  y\,\varepsilon \,\{c^{I}(a^{I},j^{I})\}(x)\rightarrow \{r^I\}(x,y)\,\varepsilon \,x\,\widetilde{\triangleleft}_{s^I,i^I,c^I}\,v^I]$$ hold. Thus just as in the previous case, we have to find sufficiently large ordinals in which all the relevant coded sets ``live''. However, we will not repeat this procedure. 
Following it, we have, by Definition \ref{fixdef}(7),  $(\mathsf{tr}(a,j,r))^I=\trt(a^I,j^I,r^I)\,\varepsilon \,a^I\;\widetilde{\triangleleft}_{s^I,i^I,c^I}\,v^I$, which means that  $\mathsf{tr}(a,j,r) \in a\triangleleft_{s,i,c}v $ is valid in the model.
\end{enumerate}
To conclude, one can show in $\mathbf{CZF}+\mathbf{REA}$ that the interpretation of $\mathbf{CT}$ is inhabited by some numeral $\mathbf{n}$. This numeral is ``almost'' a Turing machine code for the identity function, however, it also depends on the Kleene normal form predicate $T$ and the result extracting function $U$.
\end{proof}

 \begin{cor}
 Consistency of the theory $\mathbf{CZF}+\mathbf{REA}$  implies the consistency of the  theory
 \mttind  $+ \mathbf{AC} + \mathbf{CT}$.
   \end{cor}
   \begin{proof} This follows from corollary ~\ref{intttmtt} and theorem \ref{real}.
 \end{proof}
   
   \begin{cor} The theory \mttind$+\mathbf{AC}+\mathbf{CT}$ has an interpretation  in the intensional version of the type theory $\mathbf{ML}_{\mathrm{1W}}\mathbf{V}$ in  Definition 5.1 of \cite{R93} or \cite{RG94}.
   \end{cor}
   \begin{proof} This is a consequence of  corollary ~\ref{intttmtt} together with the proof of  the above Theorem \ref{real}  and  Proposition 5.3 in \cite{R93,RG94}, namely the interpretability of  $\mathbf{CZF}+\mathbf{REA}$ in $\mathbf{ML}_{\mathrm{1W}}\mathbf{V}$. \end{proof}
 
   \begin{remark} In a certain sense there is nothing special about inductively generated basic covers in that the interpretation of    \mluif\ in $\mathbf{CZF}+\mathbf{REA}$ would also work if one added further inductive types
such as generic well founded sets to   \mluif. In the same vein one could add more universes or even superuniverses (see \cite{pasu, rasu}) after beefing up the interpreting set theory by adding large set axioms.  As a consequence one can conclude that intensional Martin-L\"of type theory with some or all these type constructors added, but crucially missing the $\xi$-rule, is consistent with  the formal Church's thesis.
\end{remark}

\begin{theorem}\label{ptsmltt}\mluif\ and    $\mathbf{CZF}+\mathbf{REA}$ have the same proof-theoretic strength. \end{theorem}

\begin{proof}It follows from   \cite{R93}, Theorem 5.13, Theorem 6.9, Theorem 6.13  (or the same theorems in \cite{RG94}) together with our theorem~\ref{real} and the observation that the theory $\mathbf{IARI}$ of \cite{R93} in  Definition 6.2 can already be interpreted in \mluif\ using the interpretation of \cite{R93} in Definition 6.5.

We just recall that $\mathbf{IARI}$ is a subsystem of second order intuitionistic number theory. It has a replacement schema and an axiom of {\em inductive generation} asserting that for every binary set relation $R$ on the naturals the well-founded part of this relation is a set. 
 The interpretation for the second order variables are the propositions on the naturals with truth conditions in $U_0$.
 
The crucial step is to interpret  the axiom of {\em inductive generation} of $\mathbf{IARI}$ in \mluif. To this purpose one has to show 
that if $s\in U_0$ and $R\in \mathsf{T}(s)\times\mathsf{T}(s)\to U_0$ then the well-founded part of $R$, $\mathsf{WP}(R)$, can be given as a predicate $\mathsf{WP}(R)\in \mathsf{T}(s)\to U_0$. 
 To this end define $i\in \mathsf{T}(s)\to U_0$ by  $i(x):=s$, $v\in \mathsf{T}(s)\to U_0$ by $v(p):=\mathsf{n}_0$, $c(x,y)\in \mathsf{T}(s)\to U_0$ by $c(x,y)(z) := R(z,x)$ (so $y$ is dummy) for $x\in \mathsf{T}(s)$ and $y\in \mathsf{T}(s)$. 
 Now let $\mathsf{WP}(R)(a):=a\triangleleft_{s,i,c}v$ for $a\in \mathsf{T}(s)$.
 Then it follows that $a$ is in the well-founded part exactly when $\mathsf{WP}(R)(a)$ is inhabited.
 To see this, suppose we have a truth maker $r$ for $(\Pi x\in \mathsf{T}(s))(R(x,a) \to \mathsf{WP}(R)(x))$.
Then  $r\in (\Pi x\in \mathsf{T}(s))(x\,\epsilon\, c(a,a) \to x\triangleleft_{s,i,c}v)$, hence $\mathsf{tr}(a,a,r)\in a\triangleleft_{s,i,c}v$
by ($\mathrm{tr}\mbox{-}\triangleleft)$, whence $\mathsf{tr}(a,a,r)\in \mathsf{WP}(R)(a)$. 
 Thus $\mathsf{WP}(R)$ satisfies the appropriate closure properties characterizing the well-founded part of $R$. The  pertaining induction principle is then a consequence  of $(\mathrm{ind}\mbox{-}\triangleleft)$. 
 \end{proof}
%
%
%

\begin{remark}
 The proof of theorem~\ref{ptsmltt}
 does not work if we substitute \mluif\ with \mttind.
 The reason is that the axiom of replacement in $\mathbf{IARI}$ does not appear
 to be interpretable in \mttind\ and to establish the exact proof theoretic strength of \mttind\  is left to future work.
  \end{remark}
 \subsection*{Conclusions}
                      In the future we aim to  further extend the realizability semantics presented here to model \mfind\ enriched with coinductive definitions
                      capable of representing generated Positive Topologies
                      in \cite{somepoint}.

                     Further goals would be to study the consistency strength
                      of \mttind\ or of  \mtt\  extended with specific inductive formal topologies 
                      such as that of the Baire space.

\subsection*{Acknowledgments} We would like  to thank the anonymous  referees very much for their keenly perceptive comments.
The first author acknowledges very helpful discussions with  F. Ciraulo, P. Martin-L{\"o}f, G. Sambin and T. Streicher. The third author was supported by a grant from the John Templeton Foundation
(``A new dawn of intuitionism: mathematical and philosophical advances," ID 60842) and a grant from the University of Padova  for a research sojourn at its Department of Mathematics.

\end{document}